\documentclass[12pt]{amsart}
\usepackage{setspace}
\usepackage{amsmath, amssymb,amsthm, amsfonts, latexsym}
\usepackage{color}

\newtheorem{thm}{Theorem}

\newtheorem{prop}{Proposition}

\newtheorem{df}{Definition}

\theoremstyle{remark}
\newtheorem{remark}{Remark}[section]

\newcommand{\Ric}{\mbox{Ric}}
\newcommand{\R}{\mathbb R}

\numberwithin{equation}{section}
\newcommand{\be}{\begin{equation}}
\newcommand{\ee}{\end{equation}}
\def\p{\partial}

\def\la{\langle}
\def\ra{\rangle}
\def\lf{\left}
\def\ri{\right}
\def\Pi{\displaystyle{\mathbb{II}}}
\def\Ric{\text{\rm Ric}}

\def\vh{\vspace{.2cm}}

\def\mS{\mathbb{S}}

\def\bee{\begin{equation*}}
\def\eee{\end{equation*}}

\def\dels{\Delta_{_\Sigma}}

\def\nabs{\nabla_{_\Sigma} }

\def\H{\mathbb{H}}
\def\mS{\mathbb{S}}

\def\Hnk{\H^n (\kappa) }

\def\Spk{\mS^n_+ ( \kappa ) }

\begin{document}

\title{A functional inequality on the boundary of static manifolds}

\author{Kwok-Kun Kwong${}^\#$}
\address[Kwok-Kun Kwong]{Department of Mathematics,  National Cheng Kung University, Tainan City 70101, Taiwan}
\email{kwong@mail.ncku.edu.tw}
\thanks{${}^\#$Research partially supported by Ministry of Science and Technology in Taiwan under grant MOST103-2115-M-006-016-MY3.}

\author{Pengzi Miao$^*$}
\address[Pengzi Miao]{Department of Mathematics, University of Miami, Coral Gables, FL 33146, USA.}
\email{pengzim@math.miami.edu}

\thanks{$^*$Research partially supported by Simons Foundation Collaboration Grant for Mathematicians \#281105.}

\date{}

\begin{abstract}
On the boundary of a compact Riemannian manifold $(\Omega, g)$ whose metric $g$ is static, we establish a functional inequality
involving the static potential of $(\Omega, g)$, the second fundamental form and the mean curvature of the boundary $\p \Omega$
 respectively.
\end{abstract}

\maketitle

\markboth{Kwok-Kun Kwong and Pengzi Miao}
{A functional inequality on the boundary of   static manifolds}

\section{introduction and statement of results}

The research in this paper is largely motivated by the following result
concerning a functional inequality on the boundary of bounded domains  in the Euclidean space $\R^n$,
proved in  \cite[Corollary 3.1]{MiaoTamXie11}.

\begin{thm}[\cite{MiaoTamXie11}]  \label{thm-E}
Let $ \Omega \subset \R^n  $ be a  bounded domain  with smooth boundary $\Sigma$.
Let $H$ and $\Pi$ be the mean curvature and the second fundamental form of $\Sigma  $
 with respect to the outward normal respectively.
If $H>0$, then
\be \label{eq-2nd-var}
\int_\Sigma \lf[ \frac{ ( \dels \eta )^2 }{ H }   - \Pi  ( \nabs \eta ,\nabs \eta )\ri] d \sigma \ge 0
\ee
for any smooth function $  \eta $ on $\Sigma$. Here  $ \nabs $,  $ \dels $ denote  the
gradient, the Laplacian on $ \Sigma$ respectively, and $ d \sigma $ is the volume form on $ \Sigma $.
Moreover, equality in \eqref{eq-2nd-var} holds for some $ \eta $ if and only if
 $ \eta = a_0 + \sum_{i=1}^n a_i x_i $ for some constants $a_0, a_1, \ldots, a_n $. Here $\{ x_1, \ldots, x_n\}$
 are the standard coordinate functions  on $\R^n$.
\end{thm}

When $n=3$ and $\Sigma$ is  convex,
 it is known (\cite{MiaoTamXie11})
 that the functional on the left side of \eqref{eq-2nd-var}
represents the second variation along $\eta$ of the Wang-Yau quasi-local energy (\cite{WangYau-PRL, WangYau08})
at the  $2$-surface $ \Sigma $,
lying in the time-symmetric slice $\R^3 = \{ t = 0 \}$,
in the Minkowski spacetime $ \R^{3,1}$.
Thus,  \eqref{eq-2nd-var} can be relativistically interpreted as
the stability inequality of the  Wang-Yau  energy at  $ \Sigma$.
The general case of  such a stability inequality 
is implied by results 
in \cite{CWY, WangYau08} for a   closed,  embedded,  spacelike $2$-surface in $ \R^{3,1}$ that  projects to a convex $2$-surface along some timelike direction.

In this paper, adopting a Riemannian geometry point of view,
 we generalize Theorem \ref{thm-E}
to hypersurfaces that are boundaries of bounded domains
in a simply connected space form.
More generally, we give    an analogue of \eqref{eq-2nd-var}
on the boundary of  compact Riemannian manifolds whose metrics
 are {\em static}  (see  Definition \ref{df-static}).

First, we fix some notations. Given a constant $\kappa > 0$,
let $\Hnk$ and  $\Spk$   denote an $n$-dimensional hyperbolic space  of constant sectional curvature $-\kappa$
and  an $n$-dimensional  open  hemisphere  of constant sectional curvature $\kappa $ respectively.

\begin{thm}\label{thm-EHS}
 Suppose $(M , g)$ is
 one of  $\mathbb{R}^n$,  $\Hnk$ and  $\mathbb{S}^n_+(\kappa)$.
Let $V $ be  the  positive  function on $M$  given by
\be \label{eq-V-model}
V=
  \begin{cases}
    1, \quad &  \mathrm{if} \ (M, g) =\mathbb{R}^n,\\
    \cosh \sqrt\kappa r, \quad &   \mathrm{if} \ (M, g) =\mathbb{H}^n(\kappa), \\
   \cos \sqrt\kappa r, \quad &  \mathrm{if} \ (M, g) =\mathbb{S}^n_+(\kappa),
  \end{cases}
\ee
  where $r$ is   the distance function  from a fixed point $p$ on $(M, g)$.
  When $ (M, g) = \Spk$, $p$ is chosen to be  the center of $\Spk $ so that $V>0$ on $M$.
  Given a  bounded domain $\Omega \subset M $ with smooth boundary $\Sigma$,
  let $H$ and $\Pi$ be the mean curvature and the second fundamental form of $\Sigma$ respectively.
If $H > 0 $, then  for any smooth  function $\eta$ on $\Sigma$,
 \be \label{eq-EHS}
\begin{split}
 & \  \int_\Sigma V \lf[  \frac{ \lf[ \dels \eta + (n-1) k  \eta  \ri]^2 }{H}  -  \Pi (\nabs \eta, \nabs \eta)    \ri] d \sigma \\
 \ge & \   \int_\Sigma \frac{\p V}{\p \nu} \lf[ | \nabs \eta |^2 - (n-1) k \eta^2 \ri]  d \sigma .
\end{split}
\ee
  Here $k = 0 $ or $\pm \kappa$ is    the sectional curvature of $(M,g)$.
Moreover,
equality in \eqref{eq-EHS} holds if and only if $\eta$ is the restriction of a function
\be \label{eq-equal-model}
u =
  \begin{cases}
   a_0 +  \sum_{i=1}^n a_i x_i,  & \mathrm{if} \ (M, g) =\mathbb{R}^n,\\
    a_0 t + \sum_{i=1}^n a_i x_i,  & \mathrm{if} \ (M, g)   =\mathbb{H}^n(\kappa), \\
  a_0 x_0 + \sum_{i=1}^n a_i x_i,   &  \mathrm{if} \ (M, g) =\mathbb{S}^n_+(\kappa) .
  \end{cases}
\ee
Here $a_0, \ldots, a_n $ are arbitrary constants,  $\Hnk$ is identified with
$$ \lf\{ (t, x_1, \ldots, x_n) \in \R^{n,1} \ | \ - t^2 + \sum_{i=1}^n x_i^2 = - \frac{1}{\kappa},  \ t >  0 \ri\} $$
in the $(n+1)$-dimensional Minkowski space $\R^{n,1}$
and
$\Spk$ is identified with
$$ \lf\{ (x_0, x_1, \ldots, x_n) \in \R^{n+1} \ | \sum_{i=0}^n x_i^2 = \frac{1}{\kappa}  , \  x_0   >  0 \ri\}  $$
in the $(n+1)$-dimensional Euclidean space $\R^{n+1}$.
\end{thm}

The standard metrics  on $\R^n$, $ \Hnk$, $\Spk$ are all examples of static metrics
which admit a {\em positive}  {static potential}.  We recall  the following definition from \cite{Corvino}:

\begin{df} [\cite{Corvino}] \label{df-static}
A Riemannian metric $g$ on a manifold $M$ is called {static} if the linearized scalar
curvature map  at $g$ has a nontrivial cokernel,
i.e. if there exists a nontrivial function $f$ on $M$ such that
\be \label{eq-static-c}
- (\Delta f) g  + \nabla^2 f - f \Ric  = 0 .
\ee
Here $\nabla^2$, $\Delta $ and $\Ric$ denote   the Hessian, the Laplacian and the Ricci curvature of $g$ respectively.
\end{df}

On a connected $(M,g) $ of dimension $n$, the space of functions $f$ satisfying \eqref{eq-static-c}
has dimension at most $n+1$ (cf. \cite[Corollary 2.4]{Corvino}).
When $g$ is static on $M$,
a nontrivial solution $f$ to \eqref{eq-static-c}
is  called a {\em static potential} of $(M,g)$.

It is known that  a static metric  necessarily has constant scalar curvature (cf. \cite[Proposition 2.3]{Corvino}).
Indeed, direct calculation shows that  $(M ,g )$ is static with a positive static potential $f$ if and only if the Lorentz warped product
$ \bar{g} = - f^2 dt^2 + g $ satisfies $ \Ric (\bar{g} ) = \frac{R}{n-1} \bar{g} $
where $R$ is the scalar curvature of $g$ (cf. \cite[Proposition 2.7]{Corvino}).
This interpretation  explains why static metrics have been  widely studied in the field of  mathematical  relativity
(see  e.g. \cite{Bunting-Masood, Anderson, Corvino, Chrusciel98, Beig-Schoen,  C-G10, M-M13}).

Our next theorem generalizes Theorem \ref{thm-E} to the boundary of a compact  Riemannian manifold  whose metric is static.

\begin{thm}  \label{thm-static}
Suppose $g$ is  a static metric on an $n$-dimensional compact manifold $\Omega$  with boundary $\Sigma$
and $V$ is a positive static potential on $(\Omega, g)$.
Let $ H$, $ \Pi$ be the mean curvature, the second fundamental form of $\Sigma  $ in $(\Omega, g)$ respectively.
If $H> 0$, then
 \be \label{eq-ineq-static}
\begin{split}
 & \  \int_\Sigma V \lf[  \frac{ \lf[ \dels \eta + (n-1) k  \eta  \ri]^2 }{H}  -  \Pi (\nabs \eta, \nabs \eta)    \ri] d \sigma \\
 \ge & \   \int_\Sigma \frac{\p V}{\p \nu} \lf[ | \nabs \eta |^2 - (n-1) k \eta^2 \ri]  d \sigma
\end{split}
\ee
for any function $\eta$ on $\Sigma$.  Here  $ k \le 0 $ is a nonpositive constant satisfying
$ \Ric \ge (n-1) k g $.
Moreover, equality holds  only if
\begin{enumerate}
\item[(i)] $ k = 0 $ and  $\eta $ is the boundary value of
a function $u $  on $(\Omega, g)$ satisfying
$
\nabla^2 u   =  0 .
$
\end{enumerate}
or
\begin{enumerate}
\item[(ii)] $ k < 0 $, $ g $ is Einstein, i.e. $ \Ric = (n-1) k g$, and
$\eta $ is the boundary value of
a function $u $  on $(\Omega, g)$ satisfying
$
\nabla^2 u + k u g  =  0 .
$
\end{enumerate}
\end{thm}

In Theorem \ref{thm-static}, the fact that $ k $ is taken as a nonpositive lower bound of the Ricci curvature of $g$
is restricted   by the method of our proof (cf.  Remark \ref{rmk-k}).
Thus, if $g$ has positive Ricci curvature, \eqref{eq-ineq-static} is always a strict inequality.
However, in this case,  if in addition that $g$ is Einstein, then $k$ can be chosen to be positive
and \eqref{eq-ineq-static} is sharp (cf. Remark \ref{rmk-Einstein}).

If the metric $g$ is not static, we also give an inequality similar to that in Theorem \ref{thm-static} but under more stringent assumptions
on the boundary and  the interior curvature  (see Theorem \ref{thm-sectional}).

\section{proof of Theorems 2 and 3}
Theorem \ref{thm-E} was derived in \cite{MiaoTamXie11} as an application of Reilly's formula  \cite{Reilly77}.
(A different generalization of Theorem \ref{thm-E} was given in \cite{MiaoWang14}, again
by making use of Reilly's formula.)
To prove Theorem \ref{thm-EHS} and \ref{thm-static},
we make use
of the following  weighted Reilly's formula, recently derived by  Qiu and Xia in \cite[Theorem 1.1]{Qiu-Xia}.

\begin{prop}[\cite{Qiu-Xia}] \label{prop-QX}
Let $(\Omega, g)$ be an $n$-dimensional, compact Riemannian manifold with  boundary $\Sigma$.
Given two functions $f$, $V$ on $\Omega$ and a constant $K$, one has
\be \label{eq-QX}
\begin{split}
&  \ \int_\Omega V \lf[ ( \Delta f + K n f )^2 -  \lf| \nabla^2 f + K f g \ri|^2 \ri] d v \\
= & \  \int_\Omega \lf[ \nabla^2 V -  (\Delta V) g  - 2 (n-1) K V g + V \Ric \ri](\nabla f , \nabla f ) d v \\
&\ + (n-1) K \int_\Omega  ( \Delta V + n K V ) f^2 d v
 + \int_\Sigma \frac{\p V}{\p \nu} \lf[ | \nabs f |^2 - (n-1) K f^2 \ri]  d \sigma \\
& \ + \int_\Sigma V \lf[ 2 \lf( \frac{\p f}{\p \nu} \ri) \dels f  + H   \lf( \frac{\p f}{\p \nu} \ri)^2
+ \Pi (\nabs f, \nabs f)  + 2 (n-1) K  \lf( \frac{\p f}{\p \nu} \ri)  f  \ri] d \sigma .
\end{split}
\ee
\end{prop}

For readers' convenience, we include a proof of \eqref{eq-QX} below.

\begin{proof}
Direct calculation gives
\be \label{eq-qx-1}
\begin{split}
 \frac12 \Delta ( V | \nabla f |^2)
= & \   \frac12 ( \Delta V ) | \nabla f |^2 + \frac12  V  \Delta | \nabla  f|^2
+  \la \nabla V, \nabla | \nabla f |^2 \ra  .
\end{split}
\ee
The  integral of $  \la \nabla V, \nabla | \nabla f |^2 \ra $ can be written  as
\be \label{eq-qx-2}
\begin{split}
& \ \int_\Omega  \la \nabla V, \nabla | \nabla f |^2 \ra d v \\
= & \
\frac32 \int_\Omega  \la \nabla V, \nabla | \nabla f |^2 \ra
-  \int_\Omega \nabla^2 f ( \nabla V , \nabla f ) d v  \\
= & \  - \frac32 \int_\Omega ( \Delta V) | \nabla f |^2 d v
+ \frac32 \int_\Sigma \frac{ \p V}{\p \nu} | \nabla f |^2 d \sigma
 -\int_\Sigma \la \nabla V , \nabla f \ra \frac{\p f}{\p \nu}  d \sigma   \\
& \
 +  \int_\Omega \nabla^2 V ( \nabla f , \nabla f ) d v  + \int_\Omega  \la \nabla  V, \nabla  f \ra \Delta f d v .
\end{split}
\ee
It follows from  \eqref{eq-qx-1}, \eqref{eq-qx-2} and the Bochner formula that
\bee \label{eq-qx-3}
\begin{split}
& \ \frac12 \int_\Sigma \frac{\p }{\p \nu} \lf( V | \nabla f |^2 \ri) d \sigma
-  \int_\Omega V \lf[ | \nabla^2 f |^2 + \Ric (\nabla f , \nabla f ) + \la \nabla \Delta f , \nabla f \ra  \ri] d v
\\
= & \ - \int_\Omega ( \Delta V) | \nabla f |^2
+ \frac32 \int_\Sigma  \frac{\p V}{\p \nu}   | \nabla f |^2 -\int_\Sigma \la \nabla V , \nabla f \ra \frac{\p f}{\p \nu}  d \sigma   \\
& \
 +  \int_\Omega \nabla^2 V ( \nabla f , \nabla f ) d v  + \int_\Omega  \la \nabla  V, \nabla  f \ra \Delta f d v .
\end{split}
\eee
Using the fact
\be
\begin{split}
 \frac12 \frac{\p }{\p \nu} | \nabla f |^2
= & \ \la \nabs f , \nabs \lf( \frac{\p f}{\p \nu} \ri)  \ra - \Pi (\nabs f, \nabs f ) \\
& \ +  \frac{\p f}{\p \nu} \lf( \Delta f - \dels f - H \frac{\p f}{\p \nu} \ri)
\end{split}
\ee
and
\be
\begin{split}
 \int_\Omega V \la \nabla \Delta f , \nabla f \ra d v
& \ =  - \int_\Omega V ( \Delta f )^2 d v  - \int_\Omega  \la \nabla V, \nabla f \ra \Delta f d v \\
& \ +   \int_\Sigma V (\Delta f ) \frac{\p f}{\p \nu} d \sigma  ,
\end{split}
\ee
we   have
\be \label{eq-qx-5}
\begin{split}
& \  \int_\Sigma  V \lf[ - \Pi (\nabs f, \nabs f )
 +  \frac{\p f}{\p \nu} \lf(  - 2\dels f - H \frac{\p f}{\p \nu} \ri) \ri]  d \sigma
 -  \int_\Sigma  \frac{\p V}{\p \nu}   | \nabs f |^2    d \sigma   \\
= & \  \int_\Omega V \lf[ | \nabla^2 f |^2 - ( \Delta f )^2 \ri] + \lf[ V \Ric - ( \Delta V) g  +  \nabla^2 V \ri] ( \nabla f , \nabla f ) d v   ,
\end{split}
\ee
where we also  made the use of
\bee
\int_\Sigma  V  \la \nabs f , \nabs \lf( \frac{\p f}{\p \nu} \ri)  \ra
+ \la \nabs V, \nabs f \ra \frac{\p f}{\p \nu} d \sigma
= - \int_\Sigma V (\dels f ) \frac{\p f}{\p \nu} d \sigma
\eee
and
$
| \nabla f |^2 = \lf(  \frac{\p f}{\p \nu} \ri)^2 + | \nabs f |^2
$ along $\Sigma$.
Now \eqref{eq-QX} follows from \eqref{eq-qx-5} and  the fact
\bee \label{eq-qx-6}
\begin{split}
& \ \int_\Omega V   \lf[ | \nabla^2 f  |^2  -   ( \Delta f )^2 \ri]  dv \\
= & \
\int_\Omega V   \lf[  | \nabla^2 f + K f g  |^2 -  ( \Delta f + n K f )^2 \ri]   dv  +
 (n-1) K \int_\Omega  n  K V f^2    dv \\
 & \  +
 (n-1) K \lf[   \int_\Sigma  \lf( 2 V f \frac{\p f}{\p \nu }  -  f^2 \frac{\p V}{\p \nu}  \ri) d \sigma
+  \lf( \int_\Omega (\Delta V) f^2 - 2 V | \nabla f |^2  \ri) d v\ri] .
\end{split}
\eee
This completes the proof.
\end{proof}

\begin{remark}
Formula \eqref{eq-QX} reduces to Reilly's formula (\cite[equation(14)]{Reilly77}) when $V=1$ and $K=0$.
\end{remark}

Motivated by  equation \eqref{eq-static-c} in Definition \ref{df-static} of static metrics,
we  can rewrite formula
\eqref{eq-QX} as
\be \label{eq-QX-r}
\begin{split}
&  \ \int_\Omega V \lf[ ( \Delta f + K n f )^2 -  \lf| \nabla^2 f + K f g \ri|^2 \ri] d v \\
= & \  \int_\Omega \lf[ \nabla^2 V -  (\Delta V) g  - V \Ric \ri](\nabla f , \nabla f ) d v
  +  2  \int_\Omega V  \lf[  \Ric -  (n-1) K g  \ri](\nabla f , \nabla f ) d v \\
&\ + (n-1) K \int_\Omega  ( \Delta V + n K V ) f^2 d v
 + \int_\Sigma \frac{\p V}{\p \nu} \lf[ | \nabs f |^2 - (n-1) K f^2 \ri]  d \sigma \\
& \ + \int_\Sigma V \lf[ 2 \lf( \frac{\p f}{\p \nu} \ri) \dels f  + H   \lf( \frac{\p f}{\p \nu} \ri)^2
+ \Pi (\nabs f, \nabs f)  + 2 (n-1) K  \lf( \frac{\p f}{\p \nu} \ri)  f  \ri] d \sigma .
\end{split}
\ee
It is the second line in \eqref{eq-QX-r} that prompts one  to apply Proposition \ref{prop-QX}
to domains in a static manifold.

\begin{proof}[Proof of Theorem \ref{thm-static}]
As $ k \le 0 $, given any nontrivial $\eta $ on $\Sigma$, there exists a  unique solution $u$ to
\be \label{eq-u-extension}
\left\{
\begin{array}{rcl}
\Delta u + n k  u & = & 0 \ \ \ \mathrm{on} \ \Omega \\
u & = & \eta \ \ \ \mathrm{at} \ \Sigma  .
\end{array}
\right.
\ee
On the other hand, taking trace of  \eqref{eq-static-c} gives
\be \label{eq-static-trace}
 \Delta V +  \frac{ R }{n-1} V = 0 ,
\ee
where  $ R $ is the scalar curvature of $g$ (which is a constant).
Plug  this $V$, together with  $f = u $ and $K = k$  in \eqref{eq-QX}, using \eqref{eq-static-c},  \eqref{eq-QX-r}
and \eqref{eq-static-trace},
we have
\be \label{eq-QX-r-ap}
\begin{split}
&  \ -  \int_\Omega V  \lf| \nabla^2 u + k u g \ri|^2  d v \\
= & \
   2  \int_\Omega V  \lf[  \Ric -  (n-1) k g  \ri](\nabla u , \nabla u ) d v
+  k \lf[  n (n-1) k  - R  \ri]  \int_\Omega   V  u^2 d v \\
& \  + \int_\Sigma \frac{\p V}{\p \nu} \lf[ | \nabs \eta |^2 - (n-1) k \eta^2 \ri]  d \sigma \\
& \ + \int_\Sigma V \lf[ 2 \lf( \frac{\p u}{\p \nu} \ri) \dels \eta  + H   \lf( \frac{\p u}{\p \nu} \ri)^2
+ \Pi (\nabs \eta, \nabs \eta)  + 2 (n-1) k  \lf( \frac{\p u }{\p \nu} \ri)  \eta  \ri] d \sigma .
\end{split}
\ee
Since $V>0$, $ \Ric \ge (n-1) k g $,  $ R \ge n(n-1) k$ and $ k \le 0 $, \eqref{eq-QX-r-ap} implies
\be \label{eq-app-K-nonp}
\begin{split}
  & \ \int_\Sigma V \lf\{  \frac{ \lf[ \dels \eta + (n-1)k \eta  \ri]^2 }{H}   -  \Pi (\nabs \eta, \nabs \eta) \ri\} d \sigma \\
  \ge & \  \int_\Omega V  \lf| \nabla^2 u + k u g \ri|^2  d v +   \int_\Sigma \frac{\p V}{\p \nu} \lf[ | \nabs \eta |^2 - (n-1) K \eta^2 \ri]  d \sigma \\
& \ + \int_\Sigma V \lf[ \sqrt{H}   \lf( \frac{\p u}{\p \nu} \ri) +
 \frac{  \dels \eta + (n-1)k \eta  }{\sqrt{H}} \ri]^2   d \sigma .
\end{split}
\ee
It follows from \eqref{eq-app-K-nonp} that
\be \label{eq-app-K-nonp-1}
\begin{split}
  & \ \int_\Sigma V \lf\{  \frac{ \lf[ \dels \eta + (n-1)k \eta  \ri]^2 }{H}   -  \Pi (\nabs \eta, \nabs \eta) \ri\} d \sigma \\
  \ge & \   \int_\Sigma \frac{\p V}{\p \nu} \lf[ | \nabs \eta |^2 - (n-1) K \eta^2 \ri]  d \sigma .
\end{split}
\ee
Moreover, by \eqref{eq-QX-r-ap},  equality in \eqref{eq-app-K-nonp-1} holds only if
\begin{align}
k \lf[ n(n-1) k - R \ri] = & \ 0 , \label{eq-k-R}\\
\nabla^2 u + k u g  = &  \ 0   ,
\label{eq-interior} \\
 H  \lf( \frac{\p u}{\p \nu} \ri) +
  \dels \eta + (n-1)k \eta   =  &  \ 0
  \label{eq-bdry} .
\end{align}
Condition \eqref{eq-k-R} implies either $ k= 0 $ or $ R = n(n-1) k $.
In the later case, it follows from $ \Ric \ge (n-1) k g $ that $\Ric = (n-1)k g$, i.e.
$g$ is Einstein.
 We also note that \eqref{eq-bdry} in fact follows from \eqref{eq-interior}.
 This is because, if \eqref{eq-interior} holds,
then at $\Sigma$,
\be
 \Delta u =  \dels u + H \frac{\p u}{\p \nu} + \nabla^2 u (\nu, \nu)
= \dels u + H \frac{\p u}{\p \nu} - k u
\ee
which  implies \eqref{eq-bdry} since  $\Delta u =  - nku $.
This proves Theorem \ref{thm-static}.
\end{proof}

\begin{remark}\label{rmk-k}
In the above proof, the assumption $k \le 0$ is essentially used in only one place, i.e.
to ensure
\be \label{eq-sign-k-R}
 k [n (n-1) k - R ] \ge 0 .
 \ee
The other use of $k \le 0 $ in the construction of $u$ is not essential   because, by another theorem of
Reilly (\cite[Theorem 4]{Reilly77}), one can  still
solve \eqref{eq-u-extension} in the case of $k>0$, provided  $(\Omega ,g )$ is not isometric to $\mS^n_+(k)$.
\end{remark}

\begin{remark} \label{rmk-Einstein}
If $ \Ric = (n-1) k g$, then
$$ k [ n(n-1) k - R] = 0 $$
regardless of the sign of $k$. Therefore,  the above proof also shows that
 inequality \eqref{eq-ineq-static} still holds if  the assumption ``$ \Ric \ge (n-1) k g $ and  $k \le 0$"
 is replaced by that  $g$ is Einstein. In this case, equality holds if and only if $\eta$ is the boundary value
 of some function $u$ that satisfies
 $ \nabla^2 u + k u g = 0 $ on $(\Omega, g)$.
\end{remark}

Theorem \ref{thm-EHS} now follows  from  Theorem \ref{thm-static} and Remark \ref{rmk-Einstein}.

\begin{proof}[Proof of Theorem \ref{thm-EHS}]
Each positive function $V$ in \eqref{eq-V-model} is a solution to \eqref{eq-static-c} when $(M, g) = \R^n$,
$\Hnk$ or $\Spk$.  Hence,  inequality \eqref{eq-EHS} follows from  \eqref{eq-ineq-static} in Theorem \ref{thm-static} and Remark \ref{rmk-Einstein}.

Suppose the equality in \eqref{eq-EHS} holds from a nontrivial $\eta$. By  Theorem \ref{thm-static} and Remark \ref{rmk-Einstein},
$\eta$ is the boundary value of a function $u$ on $(\Omega, g)$ satisfying
\be
\nabla^2 u + k u g = 0 .
\ee
Since the standard metric $g$ on $\R^n$, $\Hnk$ and $\Spk$ is also Einstein,
the static equation \eqref{eq-static-c} is equivalent  to
\be \label{eq-static-reduced}
\nabla^2 f + k f g = 0 .
\ee
Therefore, $u$ is the restriction of a static potential of $(M,g )$ to  $(\Omega, g)$.
Theorem \ref{thm-EHS} now follows from the fact that the space of solutions to \eqref{eq-static-c}  on $(M, g)$ is spanned by
\begin{align*}
\{ 1, x_1, \ldots, x_n \}, & \ \ \mathrm{when} \ (M, g) = \R^n \\
\{ t , x_1, \ldots, x_n \}, &  \ \ \mathrm{when} \  (M, g) = \Hnk  \\
\{ x_0, x_1,  \ldots, x_n \}, & \  \ \mathrm{when} \  (M, g) = \Spk .
\end{align*}
\end{proof}

\begin{remark}\label{rem: local}
By \cite{cheeger1996lower} (p. 192-194) (cf. \cite{tashiro1965complete} Theorem 2 for a related result), it is known that if $(\Omega, g)$ possesses a function $u$ with $\nabla ^2 u=-k u g$, then $g$ is locally a warped product metric in the sense that there exists a Riemannian manifold $(N^{n-1}, g_N)$ such that $g$ can be locally expressed as $dr^2 + s(r)^2 g_N$ where $s(r)$ is a function on an interval $I$. In fact, their argument (which is local) shows that $u$ can be expressed as a function of $r$ and $u(r)$ satisfies the linear ODE $u''=-ku$, and that $s(r)=u'(r)$. Also, $s=u'$ and $g_N$ are unique up to multiplicative constants. Once these have been fixed, $u$ is determined by an additive constant. For example, when $k=0$, $g$ is locally a product metric $dr^2+ g_N$.
\end{remark}

\section{A similar inequality}
When the metric is not static,  there is an inequality similar to that in
Theorem \ref{thm-static}  but under more stringent conditions on the boundary and the interior curvature.

 For a compact Riemannian manifold $\Omega$ with boundary $\Sigma$, we say it is star-shaped with respect to an interior point $p\in \Omega$
 if every point in $\Omega$ can be joined by a minimal geodesic starting from $p$.

\begin{thm}\label{thm-sectional}
  Let $(\Omega, g)$ be an $n$-dimensional compact Riemannian manifold with  boundary $\Sigma$.
  Suppose $\Sigma$ has positive mean curvature and is star-shaped with respect to an interior point $p\in \Omega$.
   Let $ \kappa >  0 $ be a  constant such that $ - \kappa $ is a lower bound of the sectional curvature of $ g$.
  Let $r=d(p, \cdot)$ and $V=\cosh \sqrt \kappa r $. Here $ d(\cdot, \cdot)$ denotes the distance function on $(\Omega, g)$.
   Then for any function $\eta$ on $\Sigma$,
  \be
  \begin{split}
    & \ \int_\Sigma  V \left[   \frac{ \left[ \Delta_\Sigma \eta -(n-1)\kappa \eta \right]^2   }{H}
    -   \Pi(\nabla_\Sigma \eta, \nabla_\Sigma \eta)   \right]  d\sigma\\
   \ge &  \
     \int_\Sigma \frac{\partial V}{\partial \nu}  \left[ |\nabla_\Sigma \eta |^2 + (n-1) \kappa\eta ^2  \right] d\sigma.
\end{split}
  \ee

  Moreover,  the equality holds   only if $\Omega$ has constant curvature $-\kappa$.
 Here $H$, $ \Pi$ are the mean curvature and the second fundamental form of $\Sigma$ respectively.

\end{thm}

\begin{proof}
By Hessian comparison, we have
$$\nabla  ^2 r \le\sqrt \kappa \coth (\sqrt \kappa r) (g-dr^2).$$
This implies
$$\nabla  ^2 V= \sqrt \kappa \sinh (\sqrt \kappa r )\nabla  ^2 r + \kappa\cosh (\sqrt \kappa r  )dr^2\le \kappa\cosh (\sqrt \kappa r )g = \kappa Vg.$$
By diagonalizing $\nabla  ^2V$, we see that
$$\Delta Vg -\nabla  ^2 V\le (n-1)\kappa Vg $$
and $\Delta V\le n\kappa V$.
This implies that, for any function $u$ on $\Omega$,
\begin{align*}
     \int_\Omega \left(V  \mathrm{Ric}+2(n-1)\kappa  V   g +\nabla  ^2 V  -\Delta Vg\right)(\nabla  u ,\nabla  u)dv\ge 0.
\end{align*}
The proof then proceeds as in Theorem \ref{thm-static}.

If the equality case holds, then as in the argument of Theorem \ref{thm-static}, we have $R=-n(n-1)\kappa$, which implies $\Omega$ has constant curvature $-\kappa$ as we assume its curvature $\ge -\kappa$.

\end{proof}

\medskip

\noindent\emph{Acknowledgements.}
PM would like to thank Shanghai Center for Mathematical Sciences  for its gracious hospitality,
during which  part of the work on this paper was carried out.
Both authors  would like to thank the anonymous referee for the very useful  comments and suggestions.

\vh 

After this paper was submitted, we learned  that an inequality that is analogous 
to \eqref{eq-EHS} in Theorem \ref{thm-EHS} 
was  established by Chen, Wang and Yau (\cite{Chen_Sanya}) in the study of quasi-local energy for spacetimes  
with a cosmological  constant. 
We want to thank Professors  Po-Ning Chen and Mu-Tao Wang for helpful discussions 
concerning \eqref{eq-EHS}.

\end{document}